\documentclass[a4paper]{amsart}

\usepackage[utf8]{inputenc}
\usepackage[english]{babel}

\usepackage[T1]{fontenc}
\usepackage{lmodern}  % Latine Modern

\usepackage{hyperref}
\usepackage{my-math-symbol}
\usepackage{my-cleveref}
\usepackage{my-equation-numbering}
\usepackage{comment}
\usepackage{tikz}
\usetikzlibrary{cd}
\usepackage[colorinlistoftodos]{todonotes}

\usepackage[non-sorted-cites,initials,alphabetic,nobysame]{amsrefs}

\DeclarePairedDelimiterX{\Hproduct}[2]{\langle}{\rangle}{#1, #2}
\DeclarePairedDelimiterX{\Lproduct}[2]{\lparen}{\rparen}{#1, #2}
\newcommand{\ball}{B}
\newcommand{\Hermform}{\begin{pmatrix} - 1 & 0 \\ 0 & I_{n + 1} \end{pmatrix}}

\AtBeginDocument{%
	\def\MR#1{}
}

\title[Kohn-Rossi cohomology]{Kohn-Rossi cohomology of spherical CR manifolds}
\author{Yuya Takeuchi}
\address{Division of Mathematics \\ Institute of Pure and Applied Sciences \\ University of Tsukuba
	\\ 1-1-1 Tennodai, Tsukuba, Ibaraki 305-8571 Japan}
\email{ytakeuchi@math.tsukuba.ac.jp, yuya.takeuchi.math@gmail.com}



\thanks{This work was supported by JSPS KAKENHI Grant Number JP21K13792.}

\begin{document}

\begin{abstract}
	We prove some vanishing theorems for the Kohn-Rossi cohomology of some spherical CR manifolds.
	To this end,
	we use a canonical contact form defined via the Patterson-Sullivan measure
	and \Weitzenbock-type formulae for the Kohn Laplacian.
	We also see that
	our results are optimal in some cases.
\end{abstract}

\maketitle

\section{Introduction}
\label{section:introduction}

It is one of the most important problems in conformal geometry
to find a good representative in a conformal class;
the Yamabe problem for example.
A conformal manifold is said to be locally conformally flat
if it is locally isomorphic to the unit sphere as a conformal manifold.
Such a manifold typically arises as the quotient of a domain in the sphere
by a Kleinian group;
this is known as a Kleinian manifold.
Nayatani~\cite{Nayatani1997} has constructed a canonical Riemannian metric on a Kleinian manifold
by using the Patterson-Sullivan measure on the limit set of a Kleinian group.
He has also computed the curvature of this metric
and applied it to study the de Rham cohomology of a compact Kleinian manifold;
see also \cite{Izeki2002}.

A CR counterpart of a locally conformally flat manifold
is a \emph{spherical CR manifold};
that is,
a CR manifold locally isomorphic to $(S^{2 n + 1}, T^{1, 0} S^{2 n + 1})$.
Such a CR manifold typically arises as follows.
Let $\Gamma$ be a torsion-free discrete subgroup of $PU(n + 1, 1)$,
the automorphism group of $(S^{2 n + 1}, T^{1, 0} S^{2 n + 1})$.
Assume that the \emph{limit set} $\Lambda_{\Gamma}$ of $\Gamma$
is strictly contained in $S^{2 n + 1}$.
Then $\Gamma$ acts on $\Omega_{\Gamma} \coloneqq S^{2 n + 1} \setminus \Lambda_{\Gamma}$ properly discontinuously,
and the quotient $M_{\Gamma} \coloneqq \Omega_{\Gamma} / \Gamma$ is a spherical CR manifold.

Nayatani~\cite{Nayatani1999}, Yue~\cite{Yue1999}, and Wang~\cite{Wang2003}
have independently constructed a canonical contact form on $M_{\Gamma}$
by using the Patterson-Sullivan measure on $\Lambda_{\Gamma}$.
Moreover,
Nayatani has also computed the Tanaka-Webster Ricci curvature of this contact form
in terms of the \emph{critical exponent} $\delta_{\Gamma}$ of $\Gamma$
and a non-negative $(1, 1)$-tensor;
see \cref{eq:formula-of-Ricci-curvature}.

The aim of this paper is to study the Kohn-Rossi cohomology $H^{p, q}_{\KR}(M_{\Gamma})$ of $M_{\Gamma}$,
which is a CR analog of the Dolbeault cohomology of a complex manifold,
by using this contact form.
To this end,
we apply \Weitzenbock-type formulae for the Kohn Laplacian obtained by \cite{Case2021-Rumin-preprint}.

We first consider $H^{0, q}_{\KR}(M_{\Gamma})$ and $H^{n + 1, q}_{\KR}(M_{\Gamma})$;
in this case,
we can simplify the condition for $\Gamma$.

\begin{theorem}
\label{thm:vanishing-of-KR-cohomology-1}
	Let $\Gamma$ be a non-elementary torsion-free discrete subgroup of $PU(n + 1, 1)$
	such that $M_{\Gamma}$ is compact and $\delta_{\Gamma} < n$.
	If $q$ is an integer satisfying $(n + 2) \delta_{\Gamma} / (2 n + 2 - \delta_{\Gamma}) < q \leq n - 1$,
	then $H^{0, q}_{\KR}(M_{\Gamma}) = H^{n + 1, n - q}_{\KR}(M_{\Gamma}) = 0$.
\end{theorem}

In addition,
we will see that
the condition for $q$ in the above theorem is optimal (\cref{prop:optimal-degree-condition}).

We next investigate the Kohn-Rossi cohomology for a general bidegree.
For each $(p, q) \in \bbZ_{\geq 0} \times \bbZ_{> 0}$ with $p + q \leq n - 1$,
we set
\begin{equation}
	m_{p, q}
	\coloneqq
	\begin{cases}
		\frac{2 (n + 1) q - 2 p}{n - p + q + 2} & (2 q \leq n + 2), \\
		\frac{(2 (n + 1) q - 2 p)(n - q + 1)}{(n - p + q) (n - q + 1) + n} & (2 q \geq n + 2).
	\end{cases}
\end{equation}
Note that $0 < m_{p, q} < n$.

\begin{theorem}
\label{thm:vanishing-of-KR-cohomology-2}
	Let $\Gamma$ be a non-elementary torsion-free discrete subgroup of $PU(n + 1, 1)$
	such that $M_{\Gamma}$ is compact.
	If $\delta_{\Gamma} < m_{p, q}$,
	then $H^{p, q}_{\KR}(M_{\Gamma}) = H^{n + 1 - p, n - q}_{\KR}(M_{\Gamma}) = 0$.
\end{theorem}

We finally give some remarks on \cref{thm:vanishing-of-KR-cohomology-1,thm:vanishing-of-KR-cohomology-2}.
We will show that $H^{p, q}_{\KR}(M_{\Gamma})$ vanishes even for $p + q = n, n + 1$ 
when $\Gamma$ is convex cocompact and $\delta_{\Gamma} < 2$
(\cref{prop:vanishing-KR-cohomology-with-small-delta}).
We will also give an example of $\Gamma$
such that $\delta_{\Gamma} = n$ and $H^{p, q}_{\KR}(M_{\Gamma}) = 0$ except $q = 0, n$
(\cref{prop:vanishing-KR-cohomology-with-n}).

This paper is organized as follows.
In \cref{section:CR-geometry},
we recall basic facts on CR manifolds.
\cref{section:KR-cohomology-and-Hodge-theory} contains a brief summary of
the Kohn-Rossi cohomology and \Weitzenbock-type formulae given by Case.
In \cref{section:complex-hyperbolic-geometry},
we summarize without proofs the relevant material on the complex hyperbolic geometry
and the Patterson-Sullivan measure.
In \cref{section:canoncal-contact-form},
we give the construction of a canonical contact form on $M_{\Gamma}$.
\cref{section:proof-of-main-theorems} is devoted to the proofs of the main theorems.
In \cref{section:concluding-remarks},
we add some remarks on the Kohn-Rossi cohomology of $M_{\Gamma}$.

\medskip

\noindent
\emph{Notation.}
We use Einstein's summation convention and assume that
lowercase Greek indices $\alpha, \beta, \gamma, \dots$ run from $1, \dots, n$.

\medskip

\section{CR geometry}
\label{section:CR-geometry}

\subsection{CR structures}
\label{subsection:CR-structures}

Let $M$ be a smooth $(2 n + 1)$-dimensional manifold without boundary.
A \emph{CR structure} is a rank $n$ complex subbundle $T^{1, 0} M$
of the complexified tangent bundle $T M \otimes \bbC$ such that
\begin{equation}
	T^{1, 0} M \cap T^{0, 1} M = 0, \qquad
	[\Gamma(T^{1, 0} M), \Gamma(T^{1, 0} M)] \subset \Gamma(T^{1, 0} M),
\end{equation}
where $T^{0, 1} M$ is the complex conjugate of $T^{1, 0} M$ in $T M \otimes \bbC$.
A typical example of CR manifolds is a real hypersurface $M$ in an $(n + 1)$-dimensional complex manifold $X$;
this $M$ has the canonical CR structure
\begin{equation}
	T^{1, 0} M
	\coloneqq T^{1, 0} X |_{M} \cap (T M \otimes \bbC).
\end{equation}
In particular,
the unit sphere $S^{2 n + 1}$ in $\bbC^{n + 1}$ has the canonical CR structure $T^{1, 0} S^{2 n + 1}$.
A CR manifold $(M, T^{1, 0} M)$ is said to be \emph{spherical}
if it is locally isomorphic to $(S^{2 n + 1}, T^{1, 0} S^{2 n + 1})$.

A CR structure $T^{1, 0} M$ is said to be \emph{strictly pseudoconvex}
if there exists a nowhere-vanishing real one-form $\theta$ on $M$
such that
$\theta$ annihilates $T^{1, 0} M$ and
\begin{equation}
	- \sqrt{-1} d \theta (Z, \overline{Z}) > 0, \qquad
	0 \neq Z \in T^{1, 0} M.
\end{equation}
We call such a one-form a \emph{contact form}.
The triple $(M, T^{1, 0} M, \theta)$ is called a \emph{pseudo-Hermitian manifold}.
Denote by $T$ the \emph{Reeb vector field} with respect to $\theta$; 
that is,
the unique vector field satisfying
\begin{equation}
	\theta(T) = 1,
	\qquad T \contr d\theta = 0.
\end{equation}
Let $(Z_{\alpha})$ be a local frame of $T^{1, 0} M$,
and set $Z_{\ovxa} \coloneqq \overline{Z_{\alpha}}$.
Then
$(T, Z_{\alpha}, Z_{\ovxa})$ gives a local frame of $T M \otimes \bbC$,
called an \emph{admissible frame}.
Its dual frame $(\theta, \theta^{\alpha}, \theta^{\ovxa})$
is called an \emph{admissible coframe}.
The two-form $d \theta$ is written as
\begin{equation}
	d \theta
	= \sqrt{- 1} \tensor{l}{_{\alpha}_{\ovxb}} \theta^{\alpha} \wedge \theta^{\ovxb},
\end{equation}
where $(\tensor{l}{_{\alpha}_{\ovxb}})$ is a positive definite Hermitian matrix.
We use $\tensor{l}{_{\alpha}_{\ovxb}}$ and its inverse $\tensor{l}{^{\alpha} ^{\ovxb}}$
to raise and lower indices of tensors.

A CR manifold $(M, T^{1, 0} M)$ is said to be \emph{embeddable}
if there exists a smooth embedding of $M$ to some $\bbC^{N}$
such that $T^{1, 0} M = T^{1, 0} \bbC^{N}|_{M} \cap (T M \otimes \bbC)$.
It is known that any closed connected strictly pseudoconvex CR manifold of dimension at least five
is embeddable~\cite{Boutet_de_Monvel1975}.

\subsection{Tanaka-Webster connection}
\label{subsection:TW-connection}

A contact form $\theta$ induces a canonical connection $\nabla$,
called the \emph{Tanaka-Webster connection} with respect to $\theta$.
It is defined by
\begin{equation}
	\nabla T
	= 0,
	\quad
	\nabla Z_{\alpha}
	= \tensor{\omega}{_{\alpha}^{\beta}} Z_{\alpha},
	\quad
	\nabla Z_{\ovxa}
	= \tensor{\omega}{_{\ovxa}^{\ovxb}} Z_{\ovxb}
	\quad
	\rbra*{ \tensor{\omega}{_{\ovxa}^{\ovxb}}
	= \overline{\tensor{\omega}{_{\alpha}^{\beta}}} }
\end{equation}
with the following structure equations:
\begin{equation}
	d \theta^{\beta}
	= \theta^{\alpha} \wedge \tensor{\omega}{_{\alpha}^{\beta}}
	+ \tensor{A}{^{\beta}_{\ovxa}} \theta \wedge \theta^{\ovxa},
	\qquad
	d \tensor{l}{_{\alpha}_{\ovxb}}
	= \tensor{\omega}{_{\alpha}^{\gamma}} \tensor{l}{_{\gamma}_{\ovxb}}
	+ \tensor{l}{_{\alpha}_{\ovxg}} \tensor{\omega}{_{\ovxb}^{\ovxg}}.
\end{equation}
The tensor $\tensor{A}{_{\alpha}_{\beta}} = \overline{\tensor{A}{_{\ovxa}_{\ovxb}}}$
is symmetric and is called the \emph{Tanaka-Webster torsion}.

The curvature form
$\tensor{\Omega}{_{\alpha}^{\beta}} \coloneqq d \tensor{\omega}{_{\alpha}^{\beta}}
- \tensor{\omega}{_{\alpha}^{\gamma}} \wedge \tensor{\omega}{_{\gamma}^{\beta}}$
of the Tanaka-Webster connection satisfies
\begin{equation}
\label{eq:curvature-form-of-TW-connection}
	\tensor{\Omega}{_{\alpha}^{\beta}}
	= \tensor{R}{_{\alpha}^{\beta}_{\rho}_{\ovxs}} \theta^{\rho} \wedge \theta^{\ovxs}
	\qquad \text{modulo $\theta, \theta^{\rho} \wedge \theta^{\sigma}, \theta^{\ovxr} \wedge \theta^{\ovxs}$}.
\end{equation}
We call the tensor $\tensor{R}{_{\alpha}^{\beta}_{\rho}_{\ovxs}}$
the \emph{Tanaka-Webster curvature}.
This tensor has the symmetry 
\begin{equation}
	\tensor{R}{_{\alpha}_{\ovxb}_{\rho}_{\ovxs}}
	= \tensor{R}{_{\rho}_{\ovxb}_{\alpha}_{\ovxs}}
	= \tensor{R}{_{\alpha}_{\ovxs}_{\rho}_{\ovxb}}.
\end{equation}
Contraction of indices gives the \emph{Tanaka-Webster Ricci curvature}
$\tensor{\Ric}{_{\rho}_{\ovxs}} \coloneqq \tensor{R}{_{\alpha}^{\alpha}_{\rho}_{\ovxs}}$
and the \emph{Tanaka-Webster scalar curvature}
$\Scal \coloneqq \tensor{\Ric}{_{\rho}^{\rho}}$.
The \emph{CR Schouten tensor} $\tensor{P}{_{\alpha}_{\ovxb}}$ is defined by
\begin{equation}
	\tensor{P}{_{\alpha}_{\ovxb}}
	\coloneqq \frac{1}{n + 2} \rbra*{\tensor{\Ric}{_{\alpha}_{\ovxb}}
		- \frac{\Scal}{2 (n + 1)} \tensor{l}{_{\alpha}_{\ovxb}}}.
\end{equation}
We define the \emph{Chern tensor} $\tensor{S}{_{\alpha}_{\ovxb}_{\rho}_{\ovxs}}$ by
\begin{equation}
	\tensor{S}{_{\alpha}_{\ovxb}_{\rho}_{\ovxs}}
	\coloneqq \tensor{R}{_{\alpha}_{\ovxb}_{\rho}_{\ovxs}}
		- \tensor{P}{_{\alpha}_{\ovxb}} \tensor{l}{_{\rho}_{\ovxs}}
		- \tensor{P}{_{\alpha}_{\ovxs}} \tensor{l}{_{\rho}_{\ovxb}}
		- \tensor{P}{_{\rho}_{\ovxb}} \tensor{l}{_{\alpha}_{\ovxs}}
		- \tensor{P}{_{\rho}_{\ovxs}} \tensor{l}{_{\alpha}_{\ovxb}},
\end{equation}
which is the trace-free part of $\tensor{R}{_{\alpha}_{\ovxb}_{\rho}_{\ovxs}}$.
It is known that $(M, T^{1, 0} M)$ is a spherical CR manifold
if and only if the Chern tensor vanishes identically when $n \geq 2$~\cite{Chern-Moser1974}.

We use the square bracket to denote antisymmetrization of indices;
for example,
\begin{equation}
	\tensor{\tau}{_{[}_{\alpha_{1}}_{\alpha_{2}}_{\ovxb_{1}}_{\ovxb_{2}}_{]}}
	= \frac{1}{2 ! 2 !}
		\rbra*{\tensor{\tau}{_{\alpha_{1}}_{\alpha_{2}}_{\ovxb_{1}}_{\ovxb_{2}}}
		- \tensor{\tau}{_{\alpha_{2}}_{\alpha_{1}}_{\ovxb_{1}}_{\ovxb_{2}}}
		- \tensor{\tau}{_{\alpha_{1}}_{\alpha_{2}}_{\ovxb_{2}}_{\ovxb_{1}}}
		+ \tensor{\tau}{_{\alpha_{2}}_{\alpha_{1}}_{\ovxb_{2}}_{\ovxb_{1}}}}.
\end{equation}
As can be seen from the above equation,
we only antisymmetrize over indices of the same type.
Moreover,
we fix contracted indices under antisymmetrization;
for example,
\begin{equation}
	\tensor{\tau}{_{[}_{\alpha_{1}}_{\beta}_{\alpha_{2}}_{]}^{\beta}}
	= \frac{1}{2 !}
		\rbra*{\tensor{\tau}{_{\alpha_{1}}_{\beta}_{\alpha_{2}}^{\beta}}
		- \tensor{\tau}{_{\alpha_{2}}_{\beta}_{\alpha_{1}}^{\beta}}}.
\end{equation}

\section{Kohn-Rossi cohomology and Hodge theory}
\label{section:KR-cohomology-and-Hodge-theory}

\subsection{Kohn-Rossi cohomology}
\label{subsection:KR-cohomology}

Let $(M, T^{1, 0} M)$ be a CR manifold.
We will denote by $\Omega^{k}_{\bbC}(M)$
the space of $\bbC$-valued $k$-forms on $M$.
Define
\begin{equation}
	F^{p} \Omega^{k}_{\bbC}(M)
	\coloneqq \Set{\omega \in \Omega^{k}_{\bbC}(M) |
		\omega(\ovZ_{1}, \dots , \ovZ_{k + 1 - p} , \cdot , \dots , \cdot) = 0,
		Z_{1}, \dots , Z_{k + 1 - p} \in T^{1, 0} M}.
\end{equation}
Note that
\begin{equation}
	\Omega^{k}_{\bbC}(M) = F^{0} \Omega^{k}_{\bbC}(M)
		\supset F^{1} \Omega^{k}_{\bbC}(M) \supset \cdots \supset F^{k} \Omega^{k}_{\bbC}(M)
		\supset F^{k + 1} \Omega^{k}_{\bbC}(M) = 0.
\end{equation}
Set
\begin{equation}
	C^{p, q}(M)
	\coloneqq F^{p} \Omega^{p + q}_{\bbC}(M) / F^{p + 1} \Omega^{p + q}_{\bbC}(M).
\end{equation}
The integrability of $T^{1, 0} M$ implies that
$d (F^{p} \Omega^{k}_{\bbC} (M)) \subset F^{p} \Omega^{k + 1}_{\bbC}(M)$.
This induces the operator
\begin{equation}
	\delbb \colon C^{p, q}(M) \to C^{p, q + 1}(M); \qquad [\omega] \mapsto [d \omega],
\end{equation}
which satisfies $\delbb^{2} = 0$.
The \emph{Kohn-Rossi cohomology} $H^{p, q}_{\KR}(M)$ of bidegree $(p, q)$ is defined by
\begin{equation}
	H^{p, q}_{\KR}(M)
	\coloneqq \frac{\Ker (\delbb \colon C^{p, q}(M) \to C^{p, q + 1}(M))}
		{\Im (\delbb \colon C^{p, q - 1}(M) \to C^{p, q}(M))}.
\end{equation}
Note that this definition has been introduced by Tanaka~\cite{Tanaka1975}*{Chapter 1.4};
see~\cite{Kohn-Rossi1965}*{Section 6} for the original definition.
If $(M, T^{1, 0} M)$ is a closed embeddable strictly pseudoconvex CR manifold,
the Kohn-Rossi cohomology satisfies the Serre duality
$H^{p, q}_{\KR}(M) \cong H^{n + 1 - p, n - q}_{\KR}(M)$~\cite{Tanaka1975}*{Theorem 7.3}.
Moreover,
$H^{p, q}_{\KR}(M)$ is finite-dimensional for $1 \leq q \leq n - 1$~\cite{Tanaka1975}*{Chapter 7.2}.

\subsection{Hodge theory}
\label{subsection:Hodge-theory}

In this subsection,
we realize the Kohn-Rossi cohomology as the cohomology of a complex of differential forms,
which is a part of the \emph{bigraded Rumin complex}%
~\cites{Rumin1994,Garfield-Lee1998,Garfield2001,Case2021-Rumin-preprint}.
Moreover,
we give the Hodge theory and \Weitzenbock-type formulae of this complex,
which plays a crucial role in the proofs of our main results.
We follow the idea of~\cite{Case2021-Rumin-preprint};
see this memoir for a thorough treatment.

Let $(M, T^{1, 0} M, \theta)$ be a pseudo-Hermitian manifold of dimension $2 n + 1$.
We denote by $\expower^{p, q}(M)$ the vector bundle
\begin{equation}
	\expower^{p, q}(M)
	= \expower^{p} (T^{1, 0} M)^{\ast} \otimes \expower^{q} (T^{0, 1} M)^{\ast}
\end{equation}
on $M$ and denote by $\Omega^{p, q}(M)$ the space of smooth sections of $\expower^{p, q}(M)$.
We call an element of $\Omega^{p, q}(M)$ a \emph{$(p, q)$-form}.
Let $(\theta, \theta^{\alpha}, \theta^{\ovxa	})$ be an admissible coframe.
To simplify notation,
we write
\begin{equation}
	\theta^{A}
	\coloneqq \theta^{\alpha_{1}} \wedge \dots \wedge \theta^{\alpha_{p}},
	\qquad
	\theta^{\ovB}
	\coloneqq \theta^{\ovxb_{1}} \wedge \dots \wedge \theta^{\ovxb_{q}},
\end{equation}
where $A = (\alpha_{1}, \dots , \alpha_{p})$ and $B = (\beta_{1}, \dots , \beta_{q})$
are multi-indices of length $p$ and $q$ respectively.
Let $A^{\prime}$ and $B^{\prime}$ be multi-indices of length $p - 1$ and $q - 1$ respectively.
We identify $A$ with $(\alpha, A^{\prime})$ and $B$ with $(\beta, B^{\prime})$
when no confusion can arise.
Any $\omega \in \Omega^{p, q} (M)$ is written as
\begin{equation}
	\omega
	= \frac{1}{p! q!} \tensor{\omega}{_{A}_{\ovB}} \theta^{A} \wedge \theta^{\ovB}
\end{equation}
with $\tensor{\omega}{_{[}_{A}_{\ovB}_{]}} = \tensor{\omega}{_{A}_{\ovB}}$.
A $(p, q)$-form $\omega$ is said to be \emph{primitive}
if $\tensor{\omega}{_{\mu}_{A^{\prime}}^{\mu}_{\ovB^{\prime}}} = 0$.
The space of primitive $(p, q)$-forms will be denoted by $P^{p, q}(M)$.

The Tanaka-Webster curvature $\tensor{R}{_{\alpha}_{\ovxb}_{\rho}_{\ovxs}}$
and the Tanaka-Webster Ricci curvature $\tensor{\Ric}{_{\alpha}_{\ovxb}}$
act on $\Omega^{p, q}(M)$ as follows:
\begin{gather}
	R \, \sharp \, \overline{\sharp} \, \omega
	\coloneqq \frac{p q}{p! q!} \tensor{R}{_{[}_{\alpha}_{\ovxb}^{\ovxn}^{\mu}}
		\tensor{\omega}{_{\mu}_{A^{\prime}}_{\ovxn}_{\ovB^{\prime}}_{]}} \theta^{A} \wedge \theta^{\ovB}, \\
	\Ric \sharp \, \omega
	\coloneqq  - \frac{p}{p! q!} \tensor{\Ric}{_{[}_{\alpha}^{\mu}}
		\tensor{\omega}{_{\mu}_{A^{\prime}}_{\ovB}_{]}} \theta^{A} \wedge \theta^{\ovB}, \\
	\Ric \overline{\sharp} \, \omega
	\coloneqq - \frac{q}{p! q!} \tensor{\Ric}{^{\ovxn}_{[}_{\ovxb}}
		\tensor{\omega}{_{A}_{\ovxn}_{\ovB^{\prime}}_{]}} \theta^{A} \wedge \theta^{\ovB}.
\end{gather}
These appear in the \Weitzenbock-type formulae used in this paper.

The contact form $\theta$ induces the pointwise Hermitian inner product
\begin{equation}
	\Hproduct{\omega}{\tau}
	\coloneqq \frac{1}{p! q!} \tensor{\omega}{_{A}_{\ovB}} \tensor{\ovxt}{^{\ovB}^{A}},
\end{equation}
where $\omega = (p! q!)^{- 1} \tensor{\omega}{_{A}_{\ovB}} \theta^{A} \wedge \theta^{\ovB}$,
$\tau = (p! q!)^{- 1} \tensor{\tau}{_{A}_{\ovB}} \theta^{A} \wedge \theta^{\ovB}$,
and $\tensor{\ovxt}{_{B}_{\ovA}} \coloneqq \overline{\tensor{\tau}{_{A}_{\ovB}}}$.
The integral of this inner product gives the $L^{2}$-inner product
\begin{equation}
	\Lproduct{\omega}{\tau}
	\coloneqq \frac{1}{n!} \int_{M} \Hproduct{\omega}{\tau} \, \theta \wedge (d \theta)^{n}
\end{equation}
if $\omega$ or $\tau$ is compactly supported.

We next introduce some differential operators acting on $\Omega^{p, q}(M)$.
The Tanaka-Webster connection induces the following two differential operators:
\begin{equation}
	\nabla_{b} \colon \Omega^{p, q}(M) \to \Omega^{1, 0}(M) \otimes \Omega^{p, q}(M);
	\quad
	\frac{1}{p! q!} \tensor{\omega}{_{A}_{\ovB}} \theta^{A} \wedge \theta^{\ovB}
	\mapsto \frac{1}{p! q!} \tensor{\nabla}{_{\gamma}} \tensor{\omega}{_{A}_{\ovB}}
		\theta^{\gamma} \otimes \theta^{A} \wedge \theta^{\ovB}
\end{equation}
and
\begin{equation}
	\ovna_{b} \colon \Omega^{p, q}(M) \to \Omega^{0, 1}(M) \otimes \Omega^{p, q}(M);
	\quad
	\frac{1}{p! q!} \tensor{\omega}{_{A}_{\ovB}} \theta^{A} \wedge \theta^{\ovB}
	\mapsto \frac{1}{p! q!} \tensor{\nabla}{_{\ovxg}} \tensor{\omega}{_{A}_{\ovB}}
		\theta^{\ovxg} \otimes \theta^{A} \wedge \theta^{\ovB}.
\end{equation}
Note that $\ovna_{b}$ is the complex conjugate of $\nabla_{b}$.
We need to introduce $\del_{b}$ and $\delbb$ also,
which correspond to differentials appearing in the bigraded Rumin complex.
For $p + q \leq n - 1$,
we define $\del_{b} \colon \Omega^{p, q}(M) \to \Omega^{p + 1, q}(M)$ by
\begin{equation}
	\del_{b} \rbra*{\frac{1}{p! q!} \tensor{\omega}{_{A}_{\ovB}} \theta^{A} \wedge \theta^{\ovB}}
	\coloneqq \frac{1}{p ! q !} \rbra*{\tensor{\nabla}{_{[}_{\alpha}} \tensor{\omega}{_{A}_{\ovB}_{]}}
		- \frac{q}{n - p - q + 1} \tensor{l}{_{[}_{\alpha}_{\ovxb}}
		\tensor{\nabla}{^{\ovxn}} \tensor{\omega}{_{A}_{\ovxn}_{\ovB^{\prime}}_{]}}}
		\theta^{\alpha A} \wedge \theta^{\ovB}
\end{equation}
and $\delbb \colon \Omega^{p, q}(M) \to \Omega^{p, q + 1}(M)$ by
\begin{equation}
	\delbb \rbra*{\frac{1}{p! q!} \tensor{\omega}{_{A}_{\ovB}} \theta^{A} \wedge \theta^{\ovB}}
	\coloneqq \frac{(- 1)^{p}}{p! q!} \rbra*{\tensor{\nabla}{_{[}_{\ovxb}} \tensor{\omega}{_{A}_{\ovB}_{]}}
		- \frac{p}{n - p - q + 1} \tensor{l}{_{[}_{\alpha}_{\ovxb}}
		\tensor{\nabla}{^{\mu}} \tensor{\omega}{_{\mu}_{A^{\prime}}_{\ovB}_{]}}}
		\theta^{A} \wedge \theta^{\ovxb \ovB};
\end{equation}
see~\cite{Case2021-Rumin-preprint}*{Proposition 5.11}.
Note that $\del_{b} \rbra*{P^{p, q}(M)} \subset P^{p + 1, q}(M)$
and $\delbb \rbra*{P^{p, q}(M)} \subset P^{p, q + 1}(M)$.
Moreover,
\begin{equation}
	0 \to P^{p, 0}(M) \stackrel{\delbb}{\longrightarrow} P^{p, 1}(M)
	\stackrel{\delbb}{\longrightarrow} \dotsb 
	\stackrel{\delbb}{\longrightarrow} P^{p, n - p - 1}(M)
	\stackrel{\delbb}{\longrightarrow} P^{p, n - p}(M) \to 0
\end{equation}
is a complex,
and one has
\begin{equation}
	H^{p, q}_{\KR}(M)
	\cong \frac{\Ker (\delbb \colon P^{p, q}(M) \to P^{p, q + 1}(M))}
		{\Im (\delbb \colon P^{p, q - 1}(M) \to P^{p, q}(M))}
\end{equation}
for $p + q \leq n - 1$;
remark that the above complex is a part of a longer complex
that computes all of the Kohn-Rossi cohomology groups~\cite{Case2021-Rumin-preprint}*{Sections 9 and 11}.

The $L^{2}$-inner product induces the formal adjoints
$\del_{b}^{\ast} \colon \Omega^{p, q}(M) \to \Omega^{p - 1, q}(M)$
and $\delbb^{\ast} \colon \Omega^{p, q}(M) \to \Omega^{p, q - 1}(M)$ for $p + q \leq n$.
These operators are written as follows~\cite{Case2021-Rumin-preprint}*{Lemma 10.12}:
\begin{gather}
	\del_{b}^{\ast} \rbra*{\frac{1}{p! q!} \tensor{\omega}{_{A}_{\ovB}} \theta^{A} \wedge \theta^{\ovB}}
	= - \frac{1}{(p - 1)! q!} \tensor{\nabla}{^{\mu}} \tensor{\omega}{_{\mu}_{A^{\prime}}_{\ovB}}
		\theta^{A^{\prime}} \wedge \theta^{\ovB}, \\
	\delbb^{\ast} \rbra*{\frac{1}{p! q!} \tensor{\omega}{_{A}_{\ovB}} \theta^{A} \wedge \theta^{\ovB}}
	= - \frac{(- 1)^{p}}{p! (q - 1)!} \tensor{\nabla}{^{\ovxn}} \tensor{\omega}{_{A}_{\ovxn}_{\ovB^{\prime}}}
		\theta^{A} \wedge \theta^{\ovB^{\prime}}.
\end{gather}
Note that $\del_{b}^{\ast} \rbra*{P^{p, q}(M)} \subset P^{p - 1, q}(M)$
and $\delbb^{\ast} \rbra*{P^{p, q}(M)} \subset P^{p, q - 1}(M)$.

The \emph{Kohn Laplacian} $\KLap$ on $\Omega^{p, q}(M)$ is given by
\begin{equation}
	\KLap
	\coloneqq \frac{n - p - q}{n - p - q + 1} \delbb \delbb^{\ast} + \delbb^{\ast} \delbb
\end{equation}
for $p + q \leq n - 1$.
We follow the definition of the Kohn Laplacian given by \cite{Case2021-Rumin-preprint}*{Definition 13.1},
which differs from other previous works,
\cites{Kohn1965,Folland-Stein1974,Tanaka1975} for example.
A primitive $(p, q)$-form $\omega$ is said to be \emph{$\delbb$-harmonic} if $\KLap \omega = 0$.
We denote by $\calH^{p, q}(M)$ the space of $\delbb$-harmonic $(p, q)$-forms on $M$.
Note that
\begin{equation}
	\calH^{p, q}(M)
	= \Set{\omega \in P^{p, q}(M) | \delbb \omega = \delbb^{\ast} \omega = 0}
\end{equation}
if $M$ is closed.
Similar to the Hodge theory on Riemannian or \Kahler manifolds,
we have the following

\begin{proposition}[\cite{Case2021-Rumin-preprint}*{Corollary 15.6}]
\label{prop:harmonic-representative}
	Let $(M, T^{1, 0} M, \theta)$ be a closed and embeddable pseudo-Hermitian manifold of dimension $2 n + 1$.
	Then there exists a canonical isomorphism $H^{p, q}_{\KR}(M) \cong \calH^{p, q}(M)$ for $p + q \leq n - 1$.
\end{proposition}

Remark that we can generalize this isomorphisms for all bidegrees;
see \cite{Case2021-Rumin-preprint}*{Section 15} for details.
We conclude this section by giving \Weitzenbock-type formulae for $\KLap$,
which play a crucial role in the proofs of our main theorems.

\begin{proposition}[\cite{Case2021-Rumin-preprint}*{Proposition 13.11 and Corollary 13.13}]
\label{prop:Weitzenbock-type-formula}
	Let $(M, T^{1, 0} M, \theta)$ be a pseudo-Hermitian manifold of dimension $2 n + 1$.
	For any $\omega \in \Omega^{p, q}(M)$ with $p + q \leq n - 1$,
	one has
	\begin{align}
		\KLap \omega
		&= \frac{q}{n} \nabla_{b}^{\ast} \nabla_{b} \omega
			+ \frac{n - q}{n} \ovna_{b}^{\ast} \ovna_{b} \omega
			- \frac{1}{n - p - q + 1} (\del_{b} \del_{b}^{\ast} + \delbb \delbb^{\ast}) \omega \\
		&\quad - R \, \sharp \, \overline{\sharp} \, \omega
			- \frac{q}{n} \Ric \sharp \, \omega - \frac{n - q}{n} \Ric \overline{\sharp} \, \omega \\
		&= \frac{(q - 1) (n - p - q)}{n (n - p - q + 2)} \nabla_{b}^{\ast} \nabla_{b} \omega
			+ \frac{(n - q + 1) (n - p - q)}{n (n - p - q + 2)} \ovna_{b}^{\ast} \ovna_{b} \omega \\
		&\quad + \frac{1}{n - p - q + 2} (\del_{b}^{\ast} \del_{b} + \delbb^{\ast} \delbb) \omega
			- \frac{n - p - q}{n - p - q + 2} R \, \sharp \, \overline{\sharp} \, \omega \\
		&\quad - \frac{(q - 1) (n - p - q)}{n (n - p - q + 2)} \Ric \sharp \, \omega
			- \frac{(n - q + 1) (n - p - q)}{n (n - p - q + 2)} \Ric \overline{\sharp} \, \omega,
	\end{align}
	where $\nabla_{b}^{\ast}$ and $\ovna_{b}^{\ast}$
	are the formal adjoints of $\nabla_{b}$ and $\ovna_{b}$ respectively.
\end{proposition}

\section{Complex hyperbolic geomerty}
\label{section:complex-hyperbolic-geometry}

In this section,
we recall some basic facts on complex hyperbolic geometry
and the Patterson-Sullivan measure;
see \cites{Goldman1999,Corlette-Iozzi1999,Kapovich2022} and references therein for more details.

\subsection{Complex hyperbolic space}
\label{subsection:complex-hyperbolic-space}

The \emph{complex hyperbolic space} of dimension $n + 1$ is the ball
\begin{equation}
	\ball_{\bbC}^{n + 1}
	\coloneqq \Set{z = (z^{1}, \dots , z^{n + 1}) \in \bbC^{n + 1}|
		\abs{z}^{2} \coloneqq \sum_{i = 1}^{n + 1} \abs{z^{i}}^{2} < 1}
\end{equation}
endowed with the complete \Kahler-Einstein form
\begin{equation}
	\omega_{\ball}
	\coloneqq - \frac{1}{2} d d^{c} \log (1 - \abs{z}^{2}),
\end{equation}
where $d^{c} = (\sqrt{- 1} / 2) (\delb - \del)$.
Denote by $d(z, w)$ the geodesic distance between $z \in \ball_{\bbC}^{n + 1}$ and $w \in \ball_{\bbC}^{n + 1}$.
Note that this satisfies
\begin{equation}
	\cosh^{2} d(z, w)
	= \frac{\abs{1 - z \cdot \ovw}^{2}}{(1 - \abs{z}^{2}) (1 - \abs{w}^{2})}.
\end{equation}
\emph{Complex geodesics} are the non-empty intersections of complex lines and $\ball_{\bbC}^{n + 1}$,
which are isometric to the hyperbolic disk.

The boundary of $\ball_{\bbC}^{n + 1}$, $S^{2 n + 1}$,
has the canonical CR structure $T^{1, 0} S^{2 n + 1}$ as noted in \cref{subsection:CR-structures}.
A canonical contact form $\theta_{0}$ on $S^{2 n + 1}$ is given by
\begin{equation}
	\theta_{0}
	\coloneqq \frac{\sqrt{- 1}}{2} \sum_{i = 1}^{n + 1} (z^{i} d \ovz^{i} - \ovz^{i} d z^{i})|_{S^{2 n + 1}}.
\end{equation}
We also endow $S^{2 n + 1}$ with a sub-Riemannian metric $d_{C}$ as follows.
For any $z, w \in S^{2 n + 1}$,
we can find a smooth path $c \colon \clcl{0}{1} \to S^{2 n + 1}$
such that $c(0) = z$, $c(1) = w$, and $\theta_{0}(c^{\prime}(t)) = 0$.
The \emph{Carnot distance} $d_{C}(z, w)$ between $z$ and $w$ is the infimum of the length of such curves.
Denote by $\dim_{H} A$ the Hausdorff dimension of $A \subset S^{2 n + 1}$ with respect to $d_{C}$.
The Carnot distance $d_{C}$ defines the standard topology of $S^{2 n + 1}$,
but the Hausdorff dimension $\dim_{H} S^{2 n + 1}$of $S^{2 n + 1}$ is equal to $2 n + 2$,
which does not coincide with its topological dimension.
The boundary of a complex geodesic is a circle in $S^{2 n + 1}$
that is transverse to the canonical contact structure on $S^{2 n + 1}$,
which is known as a \emph{chain};
see \cite{Chern-Moser1974} for a more general definition.
Note that the Hausdorff dimension of a chain is equal to $2$.

Let $U(n + 1, 1)$ be the unitary group with respect to
the Hermitian form determined by $\diag(- 1, 1, \dots , 1)$;
that is,
\begin{equation}
	U(n + 1, 1)
	\coloneqq \Set{A \in GL(n + 2, \bbC) | A^{\ast} \Hermform A = \Hermform}.
\end{equation}
This group acts on both $\ball_{\bbC}^{n + 1}$ and $S^{2 n + 1}$ by the fractional linear transformation
\begin{equation}
\label{eq:group-action}
	\begin{pmatrix}
		a & b \\
		c & D
	\end{pmatrix}
	\cdot z
	\coloneqq \frac{c + D z}{a + b z}.
\end{equation}
This action preserves the \Kahler form $\omega_{\ball}$ and the CR structure $T^{1, 0} S^{2 n + 1}$.
The map \cref{eq:group-action} is equal to the identity map if and only if
the matrix is proportional to the identity matrix.
Hence the action of $U(n + 1, 1)$ descends to that of the projective unitary group $PU(n + 1, 1)$.
Moreover,
it is known that
\begin{equation}
	\Aut(\ball_{\bbC}^{n + 1}, \omega_{\ball})
	= \Aut(S^{2 n + 1}, T^{1, 0} S^{2 n + 1})
	= PU(n + 1, 1).
\end{equation}
For each $g \in PU(n + 1, 1)$,
a positive smooth function $j_{g}$ on $S^{2 n + 1}$ is defined by
\begin{equation}
	g^{\ast} \theta_{0}
	= j_{g}^{2} \theta_{0}.
\end{equation}

\subsection{Discrete subgroups of $PU(n + 1, 1)$}
\label{subsection:discrete-subgroups}

Let $\Gamma$ be a discrete subgroup of $PU(n + 1, 1)$.
Note that $\Gamma$ is discrete if and only if
$\Gamma$ acts properly discontinuously on $\ball_{\bbC}^{n + 1}$.
We write $X_{\Gamma}$ for the quotient of $\ball_{\bbC}^{n + 1}$ by $\Gamma$.
The action of $\Gamma$ on $\ball_{\bbC}^{n + 1}$ is free
if and only if $\Gamma$ is torsion-free;
in this case,
$X_{\Gamma}$ is a smooth complex manifold.

The \emph{limit set} $\Lambda_{\Gamma}$ of a discrete subgroup $\Gamma$ of $PU(n + 1, 1)$
is the set of accumulation points in $\overline{\ball}_{\bbC}^{n + 1}$
of the $\Gamma$-orbit of any point in $\ball_{\bbC}^{n + 1}$,
which is a closed subset of $S^{2 n + 1}$.
We call $\Gamma$ \emph{elementary} if $\# \Lambda_{\Gamma} \leq 2$;
otherwise we call $\Gamma$ \emph{non-elementary}.
If $\Gamma$ is non-elementary,
then $\Lambda_{\Gamma}$ is the smallest non-empty closed $\Gamma$-invariant subset of $S^{2 n + 1}$;
see \cite{Kapovich2022}*{Proposition 2} for example.
The complement $\Omega_{\Gamma} \coloneqq S^{2 n + 1} \setminus \Lambda_{\Gamma}$ of $\Lambda_{\Gamma}$
is called the \emph{domain of discontinuity}.
This is the largest open subset of $S^{2 n + 1}$
on which $\Gamma$ acts properly discontinuously.
If $\Omega_{\Gamma}$ is non-empty,
denote by $M_{\Gamma}$ (resp.\ $\ovX_{\Gamma}$)
the quotient of $\Omega_{\Gamma}$ (resp.\ $\ball_{\bbC}^{n + 1} \cup \Omega_{\Gamma}$) by $\Gamma$.
The action of $\Gamma$ on $\Omega_{\Gamma}$ is free
if and only if $\Gamma$ is torsion-free;
in this case,
$\ovX_{\Gamma}$ is a smooth complex manifold with boundary $M_{\Gamma}$,
and $M_{\Gamma}$ is a spherical CR manifold.

Let $\Gamma$ be a discrete subgroup of $PU(n + 1, 1)$ satisfying $\# \Lambda_{\Gamma} \geq 2$.
The \emph{closed convex hull} $C_{\Gamma}$ of $\Lambda_{\Gamma}$
is the intersection of all closed convex subsets in $\ball_{\bbC}^{n + 1}$
whose boundary contains $\Lambda_{\Gamma}$.
This subset is invariant under $\Gamma$,
and we say that $\Gamma$ is \emph{convex cocompact}
if the quotient of $C_{\Gamma}$ by $\Gamma$ is compact.
This condition is equivalent to that $\ovX_{\Gamma}$ is compact~\cite{Bowditch1995};
in particular,
so is $M_{\Gamma}$.

\subsection{Critical exponent and Patterson-Sullivan measure}
\label{subsection:critical-exponent-and-PS-measure}

Let $\Gamma$ be a discrete subgroup of $PU(n + 1, 1)$
and take $z, w \in \ball_{\bbC}^{n + 1}$.
For $s > 0$,
we define
\begin{equation}
	\Phi_{s}(z, w)
	\coloneqq \sum_{g \in \Gamma} e^{- s d(z, g w)} \in \opcl{0}{+ \infty}.
\end{equation}
The \emph{critical exponent} $\delta_{\Gamma}$ of $\Gamma$ is given by
\begin{equation}
	\delta_{\Gamma}
	\coloneqq \inf \Set{s \in \opop{0}{+ \infty} | \Phi_{s}(z, w) < + \infty};
\end{equation}
note that the condition $\Phi_{s}(z, w) < + \infty$
is independent of the choice of $z, w \in \ball_{\bbC}^{n + 1}$.
It is known that $0 \leq \delta_{\Gamma} \leq 2 n + 2$
and $\delta_{\Gamma} = 0$ if and only if $\Gamma$ is elementary.
Moreover,
one has $\delta_{\Gamma} = \dim_{H} \Lambda_{\Gamma}$ if $\Gamma$ is convex cocompact%
~\cite{Corlette-Iozzi1999}*{Theorem 6.1}.

Following Patterson~\cite{Patterson1976} and Sullivan~\cite{Sullivan1979},
Corlette~\cite{Corlette1990}*{Proposition 5.1} has constructed
a probability measure $\mu_{\Gamma}$ on $\Lambda_{\Gamma}$
satisfying $g^{\ast} \mu_{\Gamma} = j_{g}^{\delta_{\Gamma}} \mu_{\Gamma}$ for any $g \in \Gamma$,
which we call the \emph{Patterson-Sullivan measure} of $\Gamma$.
Moreover,
$\mu_{\Gamma}$ coincides with
the $\delta_{\Gamma}$-dimensional Hausdorff measure with respect to $d_{C}$ up to scaling
if $\Gamma$ is non-elementary and convex cocompact~\cite{Corlette1990}*{Theorem 5.4}.

\section{Canonical contact form}
\label{section:canoncal-contact-form}

In this section,
we construct a $\Gamma$-invariant contact form on $\Omega_{\Gamma}$
by using the Patterson-Sullivan measure of $\Gamma$.
Note that this contact form has been introduced by
Nayatani~\cite{Nayatani1999}, Yue~\cite{Yue1999}, and Wang~\cite{Wang2003} independently.

Let $\varphi(z, w) \coloneqq \abs{1 - z \cdot \ovw}$.
This function satisfies
\begin{equation}
	\varphi(g z, g w)
	= j_{g}(z) j_{g}(w) \varphi(z, w)
\end{equation}
for any $g \in PU(n + 1, 1)$.
Note that $\varphi(z, w)^{- n}$ is the Green kernel of the CR Yamabe operator on $S^{2 n + 1}$ up to scaling.

Let $\Gamma$ be a non-elementary discrete group of $PU(n + 1, 1)$ with $\Omega_{\Gamma} \neq \emptyset$.
Note that $\delta_{\Gamma} > 0$.
Define a contact form $\theta_{\Gamma}$ on $\Omega_{\Gamma}$ by
\begin{equation}
	\theta_{\Gamma}
	\coloneqq \rbra*{\int_{\Lambda_{\Gamma}} \varphi(z, w)^{- \delta_{\Gamma}}
		\, d \mu_{\Gamma} (w)}^{2 / \delta_{\Gamma}} \theta_{0}.
\end{equation}
The transformation laws of $\varphi$, $\mu_{\Gamma}$, and $\theta_{0}$ yield that
the contact form $\theta_{\Gamma}$ is $\Gamma$-invariant.
In particular if $\Gamma$ is torsion-free,
$\theta_{\Gamma}$ descends to a contact form on $M_{\Gamma}$;
we use the same letter $\theta_{\Gamma}$ for this contact form by abuse of notation.

We introduce a probability measure
\begin{equation}
	\nu
	\coloneqq \rbra*{\int_{\Lambda_{\Gamma}} \varphi(z, w)^{- \delta_{\Gamma}} \, d \mu_{\Gamma}(w)}^{- 1}
		\varphi(z, \cdot)^{- \delta_{\Gamma}} \mu_{\Gamma}
\end{equation}
and a $(1, 1)$-tensor
\begin{align}
	\tensor{D}{_{\alpha}_{\ovxb}}
	&\coloneqq 2 \int_{\Lambda_{\Gamma}} \varphi_{w}^{- 2}
		(\tensor{\nabla}{_{\alpha}} \varphi_{w}) (\tensor{\nabla}{_{\ovxb}} \varphi_{w}) \, d \nu(w) \\
	&\quad - 2 \rbra*{\int_{\Lambda_{\Gamma}} \varphi_{w}^{- 1} (\tensor{\nabla}{_{\alpha}} \varphi_{w}) \, d \nu(w)}
		\rbra*{\int_{\Lambda_{\Gamma}} \varphi_{w}^{- 1} (\tensor{\nabla}{_{\ovxb}} \varphi_{w}) \, d \nu(w)},
\end{align}
where $\varphi_{w}(z) \coloneqq \varphi(z, w)$.
For any $Z \in T^{1, 0} M$,
the \Holder inequality gives
\begin{equation}
	\abs*{\int_{\Lambda_{\Gamma}} \varphi_{w}^{- 1} (Z \varphi_{w}) \, d \nu(w)}^{2}
	\leq \int_{\Lambda_{\Gamma}} \varphi_{w}^{- 2} \abs{Z \varphi_{w}}^{2} \, d \nu(w).
\end{equation}
This implies that $\tensor{D}{_{\alpha}_{\ovxb}}$ is non-negative as a Hermitian form.
Moreover,
we have

\begin{lemma}[\cite{Nayatani1999}*{Lemma 2.3}]
\label{lem:trace-of-D}
	The trace $\tensor{D}{_{\gamma}^{\gamma}}$ of $\tensor{D}{_{\alpha}_{\ovxb}}$ is positive everywhere
	unless $\Lambda_{\Gamma}$ lies properly in a chain.
	If $\Lambda_{\Gamma}$ lies properly in a chain $C$,
	then $\tensor{D}{_{\gamma}^{\gamma}}$ is positive on $S^{2 n + 1} \setminus C$
	and vanishes on $C \setminus \Lambda_{\Gamma}$.
	In both cases,
	$\tensor{D}{_{\gamma}^{\gamma}}$ is positive on an open dense subset of $\Omega_{\Gamma}$.
\end{lemma}

The Tanaka-Webster Ricci curvature $\tensor{\Ric}{_{\alpha}_{\ovxb}}$ of $\theta_{\Gamma}$ is given by
\begin{equation}
\label{eq:formula-of-Ricci-curvature}
	\tensor{\Ric}{_{\alpha}_{\ovxb}}
	= - (n + 2) \delta_{\Gamma} \tensor{D}{_{\alpha}_{\ovxb}}
		+ (2 n + 2 - \delta_{\Gamma}) \tensor{D}{_{\gamma}^{\gamma}} \tensor{l}{_{\alpha}_{\ovxb}};
\end{equation}
see \cite{Nayatani1999}*{Proposition 2.2}.
This implies that
\begin{equation}
	\Scal
	= 2 (n + 1) (n - \delta_{\Gamma}) \tensor{D}{_{\gamma}^{\gamma}},
	\qquad
	\tensor{P}{_{\alpha}_{\ovxb}}
	= - \delta_{\Gamma} \tensor{D}{_{\alpha}_{\ovxb}} + \tensor{D}{_{\gamma}^{\gamma}} \tensor{l}{_{\alpha}_{\ovxb}}.
\end{equation}
Moreover,
the Chern tensor vanishes identically since $(S^{2 n + 1}, T^{1, 0} S^{2 n + 1})$ is spherical.
Thus we have
\begin{equation}
\label{eq:formula-of-curvature}
	\tensor{R}{_{\alpha}_{\ovxb}_{\rho}_{\ovxs}}
	= - \delta_{\Gamma} (\tensor{D}{_{\alpha}_{\ovxb}} \tensor{l}{_{\rho}_{\ovxs}}
		+ \tensor{D}{_{\alpha}_{\ovxs}} \tensor{l}{_{\rho}_{\ovxb}}
		+ \tensor{D}{_{\rho}_{\ovxb}} \tensor{l}{_{\alpha}_{\ovxs}}
		+ \tensor{D}{_{\rho}_{\ovxs}} \tensor{l}{_{\alpha}_{\ovxb}})
		+ 2 \tensor{D}{_{\gamma}^{\gamma}} (\tensor{l}{_{\alpha}_{\ovxb}} \tensor{l}{_{\rho}_{\ovxs}}
		+ \tensor{l}{_{\alpha}_{\ovxs}} \tensor{l}{_{\rho}_{\ovxb}}).
\end{equation}

\begin{remark}
\label{rem:scalar-curvature}
	Yue~\cite{Yue1999}*{Theorem A} and Wang~\cite{Wang2003}*{Theorem 1.5} have stated that
	the Tanaka-Webster scalar curvature of $\theta_{\Gamma}$ is positive everywhere if $\delta_{\Gamma} < n$,
	which contradicts \cite{Nayatani1999}*{Theorem 2.4(ii)}.
	Unfortunately,
	the statement given by Yue and Wang is not true.
	Here we give a counterexample.
	Take a convex cocompact torsion-free discrete subgroup $\Gamma$ of $PU(n + 1, 1)$ such that
	$\Lambda_{\Gamma} \subset S^{1} \times \{0\}$ and $0 < \delta_{\Gamma} < 1$;
	such a $\Gamma$ can be constructed a method similar to \cite{Beardon1968}*{p.\ 480} for example.
	Note that $\Lambda_{\Gamma} \neq S^{1} \times \{0\}$
	since
	\begin{equation}
		\dim_{H} \Lambda_{\Gamma} = \delta_{\Gamma} < 1 < 2 = \dim_{H} (S^{1} \times \{0\}).
	\end{equation}
	\cite{Nayatani1999}*{Lemma 2.3} implies that $\tensor{D}{_{\alpha}_{\ovxb}} = 0$
	on $(S^{1} \times \{0\}) \setminus \Lambda_{\Gamma}$.
	In particular,
	$\Scal = 0$ there although $0 < \delta_{\Gamma} < 1 \leq n$.
\end{remark}

\section{Proof of main theorems}
\label{section:proof-of-main-theorems}

In this section,
we give the proofs of our main theorems.
Let $\Gamma$ be a non-elementary torsion-free discrete subgroup of $PU(n + 1, 1)$
such that $M_{\Gamma}$ is compact.
Similar to the action of the Tanaka-Webster curvature on $\Omega^{p, q}(M_{\Gamma})$,
we define
\begin{gather}
	D \, \sharp \, \omega
	\coloneqq  - \frac{p}{p! q!} \tensor{D}{_{[}_{\alpha}^{\mu}}
		\tensor{\omega}{_{\mu}_{A^{\prime}}_{\ovB}_{]}} \theta^{A} \wedge \theta^{\ovB}, \\
	D \, \overline{\sharp} \, \omega
	\coloneqq - \frac{q}{p! q!} \tensor{D}{^{\ovxn}_{[}_{\ovxb}}
		\tensor{\omega}{_{A}_{\ovxn}_{\ovB^{\prime}}_{]}} \theta^{A} \wedge \theta^{\ovB}, \\
	D \, \sharp \, \overline{\sharp} \, \omega
	\coloneqq \frac{p q}{p! q!} \tensor{l}{_{[}_{\alpha}_{\ovxb}} \tensor{D}{^{\ovxn}^{\mu}}
		\tensor{\omega}{_{\mu}_{A^{\prime}}_{\ovxn}_{\ovB^{\prime}}_{]}} \theta^{A} \wedge \theta^{\ovB}
\end{gather}
for $\omega = (p! q!)^{- 1} \tensor{\omega}{_{A}_{\ovB}} \theta^{A} \wedge \theta^{\ovB}
\in \Omega^{p, q}(M_{\Gamma})$.
Note that $\Hproduct{D \, \sharp \, \overline{\sharp} \, \omega}{\tau} = 0$
for any $\tau \in P^{p, q}(M_{\Gamma})$.
Since $\tensor{D}{_{\alpha}_{\ovxb}}$ is non-negative,
\begin{equation}
\label{eq:estimate-of-D}
	- \tensor{D}{_{\gamma}^{\gamma}} \Hproduct{\omega}{\omega}
	\leq \Hproduct{D \, \sharp \, \omega}{\omega}
	\leq 0,
	\qquad
	- \tensor{D}{_{\gamma}^{\gamma}} \Hproduct{\omega}{\omega}
	\leq \Hproduct{D \, \overline{\sharp} \, \omega}{\omega}
	\leq 0
\end{equation}
for any $\omega \in \Omega^{p, q}(M_{\Gamma})$.

\begin{proof}[Proof of \cref{thm:vanishing-of-KR-cohomology-1}]
	It suffices to show that $H^{0, q}_{\KR}(M_{\Gamma}) = 0$ by the Serre duality.
	Moreover,
	\cref{prop:harmonic-representative} implies that
	it is enough to prove that any $\omega \in \calH^{0, q}(M_{\Gamma})$ must be zero.
	Note that
	\begin{equation}
		\del_{b}^{\ast} \omega
		= R \, \sharp \, \overline{\sharp} \, \omega
		= \Ric \sharp \, \omega
		= 0
	\end{equation}
	since $\omega \in \Omega^{0, q}(M_{\Gamma})$.
	It follows from \cref{prop:Weitzenbock-type-formula} that
	\begin{equation}
		0
		= q \nabla_{b}^{\ast} \nabla_{b} \omega
			+ (n - q) \ovna_{b}^{\ast} \ovna_{b} \omega
			- (n - q) \Ric \overline{\sharp} \, \omega
	\end{equation}
	Taking the $L^{2}$-product with $\omega$ gives that
	\begin{align}
		0
		&= q \Lproduct{\nabla_{b} \omega}{\nabla_{b} \omega}
			+ (n - q) \Lproduct{\ovna_{b} \omega}{\ovna_{b} \omega}
			- (n - q) \Lproduct{\Ric \overline{\sharp} \, \omega}{\omega} \\
		&\geq - (n - q) \Lproduct{\Ric \overline{\sharp} \, \omega}{\omega}.
	\end{align}
	On the other hand,
	\cref{eq:formula-of-Ricci-curvature} yields that
	\begin{equation}
		\Ric \overline{\sharp} \, \omega
		= - (n + 2) \delta_{\Gamma} D \, \overline{\sharp} \, \omega
			- q (2 n + 2 - \delta_{\Gamma}) \tensor{D}{_{\gamma}^{\gamma}} \omega.
	\end{equation}
	Hence
	\begin{align}
		\Hproduct{\Ric \overline{\sharp} \, \omega}{\omega}
		&= - (n + 2) \delta_{\Gamma} \Hproduct{D \, \overline{\sharp} \, \omega}{\omega}
			- q (2 n + 2 - \delta_{\Gamma}) \tensor{D}{_{\gamma}^{\gamma}} \Hproduct{\omega}{\omega} \\
		&\leq \sbra*{(n + 2) \delta_{\Gamma} - q (2 n + 2 - \delta_{\Gamma})}
			\tensor{D}{_{\gamma}^{\gamma}} \Hproduct{\omega}{\omega}.
	\end{align}
	Thus we have
	\begin{equation}
		0
		\geq (n - q) \sbra*{q (2 n + 2 - \delta_{\Gamma}) - (n + 2) \delta_{\Gamma}}
			\int_{M_{\Gamma}} \tensor{D}{_{\gamma}^{\gamma}} \Hproduct{\omega}{\omega}
			\, \theta_{\Gamma} \wedge (d \theta_{\Gamma})^{n}.
	\end{equation}
	Now it follows from the assumption that
	\begin{equation}
		(n - q) \sbra*{q (2 n + 2 - \delta_{\Gamma}) - (n + 2) \delta_{\Gamma})}
		> 0.
	\end{equation}
	Moreover,
	$\tensor{D}{_{\gamma}^{\gamma}}$ is positive on an open dense subset by \cref{lem:trace-of-D}.
	Therefore $\Hproduct{\omega}{\omega}$ is equal to zero on an open dense subset,
	which implies $\omega = 0$ by the continuity.
\end{proof}

\begin{example}
\label{example:closed-quotient}
	Let $\Gamma_{0}$ be a torsion-free discrete subgroup of $U(k, 1)$ with $1 \leq k < n / 2$
	such that the quotient $\ball_{\bbC}^{k} / \Gamma_{0}$ is compact.
	We embed $U(k, 1)$ into $PU(n + 1, 1)$ by the composition of
	\begin{equation}
		U(k, 1) \to U(n + 1, 1);
		\qquad
		A \mapsto
		\begin{pmatrix}
			A & 0 \\
			0 & I_{n - k + 1}
		\end{pmatrix}
	\end{equation}
	and the canonical projection $U(n + 1, 1) \to PU(n + 1, 1)$.
	Denote by $\Gamma$ the image of $\Gamma_{0}$ under this embedding.
	The group $\Gamma$ is a torsion-free discrete subgroup of $PU(n + 1, 1)$
	with $\Lambda_{\Gamma} = S^{2 k - 1} \times \{0\} \subset S^{2 n + 1}$.
	The closed convex hull $C_{\Gamma}$ of $\Lambda_{\Gamma}$ coincides with $\ball_{\bbC}^{k} \times \{0\}$.
	Thus we have $C_{\Gamma} / \Gamma = (\ball_{\bbC}^{k} / \Gamma_{0}) \times \{0\}$,
	which is compact.
	Hence $\Gamma$ is convex cocompact and
	\begin{equation}
		\delta_{\Gamma}
		= \dim_{H} \Lambda_{\Gamma}
		= 2 k
		< n.
	\end{equation}
	\cref{thm:vanishing-of-KR-cohomology-1} implies that
	$H^{0, q}_{\KR}(M_{\Gamma}) = H^{n + 1, n - q}_{\KR}(M_{\Gamma}) = 0$
	for any integer $q$ with $k (n + 2) / (n - k + 1) < q \leq n - 1$.
	Note that we can identify $M_{\Gamma}$ with a principal $S^{1}$-bundle
	over $(\ball_{\bbC}^{k} / \Gamma_{0}) \times \cps^{n - k}$ as follows.
	Consider the following projection:
	\begin{equation}
		\pi \colon \Omega_{\Gamma} \to \ball_{\bbC}^{k} \times \cps^{n - k};
		\quad
		(z^{1}, \dots , z^{k}, z^{k + 1}, \dots , z^{n + 1})
		\mapsto (z^{1}, \dots , z^{k}, [z^{k + 1} : \dots : z^{n + 1}]).
	\end{equation}
	This projection satisfies
	$\pi_{\ast} T^{1, 0} \Omega_{\Gamma} = T^{1, 0} (\ball_{\bbC}^{k} \times \cps^{n - k})$.
	Moreover,
	$\Omega_{\Gamma}$ has the free $S^{1}$-action
	\begin{equation}
		e^{i \theta} \cdot (z^{1}, \dots , z^{k}, z^{k + 1}, \dots , z^{n + 1})
		= (z^{1}, \dots , z^{k}, e^{i \theta} z^{k + 1}, \dots , e^{i \theta} z^{n + 1}),
	\end{equation}
	which preserves the CR structure and commutes with the action of $\Gamma$.
	Thus we have a principal $S^{1}$-bundle
	\begin{equation}
		\pi_{\Gamma} \colon M_{\Gamma} \to Y \coloneqq (\ball_{\bbC}^{k} / \Gamma_{0}) \times \cps^{n - k}
	\end{equation}
	such that $T^{1, 0} M_{\Gamma}$ is $S^{1}$-invariant
	and $(\pi_{\Gamma})_{\ast} T^{1, 0} M_{\Gamma} = T^{1, 0} Y$.
\end{example}

The following proposition implies that
the condition for $q$ in \cref{thm:vanishing-of-KR-cohomology-1} is optimal for $n \geq 3$;
see \cref{prop:vanishing-KR-cohomology-with-small-delta} for $n = 2$.

\begin{proposition}
\label{prop:optimal-degree-condition}
	For each $n \geq 2$,
	there exists a torsion-free convex cocompact discrete subgroup $\Gamma$ of $PU(n + 1, 1)$
	such that $\delta_{\Gamma} = 2$ and $H^{0, 1}_{\KR}(M_{\Gamma}) \neq 0$.
\end{proposition}

\begin{proof}
	Take a torsion-free discrete subgroup $\Gamma_{0}$ of $U(1, 1)$
	so that $\Sigma \coloneqq \ball_{\bbC}^{1} / \Gamma_{0}$ is a closed Riemann surface of genus $\geq 2$.
	Let $\Gamma$ be as in \cref{example:closed-quotient}.
	Take a $\delb$-closed $(0, 1)$-form $\omega_{0}$ on $Y$
	such that $[\omega_{0}] \neq 0$ in $H^{0, 1}(Y)$;
	such an $\omega_{0}$ exists since $H^{0, 1}(\Sigma) \neq 0$.
	Then $\omega \coloneqq \pi_{\Gamma}^{\ast} \omega_{0} \in F^{0} \Omega^{1}_{\bbC}(M_{\Gamma})$
	and $\delbb [\omega] = 0$.
	Suppose to the contrary that $[\omega]$ defines a trivial cohomology class in $H^{0, 1}_{\KR}(M_{\Gamma})$.
	This means that there exists $f \in C^{\infty}(M_{\Gamma})$
	such that $(d f)|_{T^{0, 1} M_{\Gamma}} = \omega$.
	Taking the mean of $f$ with respect to the $S^{1}$-action,
	we may assume that $f$ is $S^{1}$-invariant
	and consider $f$ as a smooth function on $Y$.
	Then $\omega_{0} = \delb f$,
	which contradicts the assumption $[\omega_{0}] \neq 0$ in $H^{0, 1}(Y)$.
	Therefore $H^{0, 1}_{\KR}(M_{\Gamma}) \neq 0$.
\end{proof}

We also consider $H^{p, q}_{\KR}(M_{\Gamma})$ for a general bidegree.
In this case,
the assumption is rather complicated
since $R \, \sharp \, \overline{\sharp} \, \omega$ and $\Ric \sharp \, \omega$ are non-trivial.

\begin{proof}[Proof of \cref{thm:vanishing-of-KR-cohomology-2}]
	It suffices to show that $H^{p, q}_{\KR}(M_{\Gamma}) = 0$ by the Serre duality.
	Moreover,
	\cref{prop:harmonic-representative} implies that
	it is enough to prove that any $\omega \in \calH^{p, q}(M_{\Gamma})$ must be zero.
	It follows from \cref{prop:Weitzenbock-type-formula} that
	\begin{align}
		0
		&= (q - 1) \nabla_{b}^{\ast} \nabla_{b} \omega
			+ (n - q + 1) \ovna_{b}^{\ast} \ovna_{b} \omega
			+ \frac{n}{n - p - q} \del_{b}^{\ast} \del_{b} \omega \\
		&\quad - n R \, \sharp \, \overline{\sharp} \, \omega
			- (q - 1) \Ric \sharp \, \omega
			- (n - q + 1) \Ric \overline{\sharp} \, \omega.
	\end{align}
	Taking the $L^{2}$-product with $\omega$ gives that
	\begin{align}
		0
		&= (q - 1) \Lproduct{\nabla_{b} \omega}{\nabla_{b} \omega}
			+ (n - q + 1) \Lproduct{\ovna_{b} \omega}{\ovna_{b} \omega}
			+ \frac{n}{n - p - q} \Lproduct{\del_{b} \omega}{\del_{b} \omega} \\
		&\quad + \Lproduct{- n R \, \sharp \, \overline{\sharp} \, \omega
			- (q - 1) \Ric \sharp \, \omega
			- (n - q + 1) \Ric \overline{\sharp} \, \omega}{\omega} \\
		&\geq \Lproduct{- n R \, \sharp \, \overline{\sharp} \, \omega
			- (q - 1) \Ric \sharp \, \omega
			- (n - q + 1) \Ric \overline{\sharp} \, \omega}{\omega}.
	\end{align}
	On the other hand,
	\cref{eq:formula-of-Ricci-curvature,eq:formula-of-curvature} yield that
	\begin{gather}
		R \, \sharp \, \overline{\sharp} \, \omega
		= q \delta_{\Gamma} D \, \sharp \, \omega + p \delta_{\Gamma} D \, \overline{\sharp} \, \omega
			+ 2 p q \tensor{D}{_{\gamma}^{\gamma}} \omega
			- \delta_{\Gamma} D \, \sharp \, \overline{\sharp} \, \omega, \\
		\Ric \sharp \, \omega
		= - (n + 2) \delta_{\Gamma} D \, \sharp \, \omega
			- p (2 n + 2 - \delta_{\Gamma}) \tensor{D}{_{\gamma}^{\gamma}} \omega, \\
		\Ric \overline{\sharp} \, \omega
		= - (n + 2) \delta_{\Gamma} D \, \overline{\sharp} \, \omega
			- q (2 n + 2 - \delta_{\Gamma}) \tensor{D}{_{\gamma}^{\gamma}} \omega.
	\end{gather}
	Hence
	\begin{align}
		&\Hproduct{- n R \, \sharp \, \overline{\sharp} \, \omega
			- (q - 1) \Ric \sharp \, \omega
			- (n - q + 1) \Ric \overline{\sharp} \, \omega}{\omega} \\
		&= (2 q - n - 2) \delta_{\Gamma} \Hproduct{D \, \sharp \, \omega}{\omega}
			+ ((n + 2) (n - q + 1) - n p) \delta_{\Gamma} \Hproduct{D \, \overline{\sharp} \, \omega}{\omega} \\
		&\quad + \sbra*{2 ((n + 1) q - p) (n - q + 1) - (p (q - 1) + q (n - q + 1)) \delta_{\Gamma}}
			\tensor{D}{_{\gamma}^{\gamma}} \Hproduct{\omega}{\omega}.
	\end{align}
	Here we have
	\begin{equation}
		(n + 2) (n - q + 1) - n p
		= n (n - p - q + 1) + 2 (n - q + 1)
		> 0.
	\end{equation}
	This and \cref{eq:estimate-of-D} imply
	\begin{align}
		&\Hproduct{- n R \, \sharp \, \overline{\sharp} \, \omega
			- (q - 1) \Ric \sharp \, \omega
			- (n - q + 1) \Ric \overline{\sharp} \, \omega}{\omega} \\
		&= (2 q - n - 2) \delta_{\Gamma} \Hproduct{D \, \sharp \, \omega}{\omega}
			+ ((n + 2) (n - q + 1) - n p) \delta_{\Gamma}
			\Hproduct{\tensor{D}{_{\gamma}^{\gamma}} \omega + D \, \overline{\sharp} \, \omega}{\omega} \\
		&\quad + (n - q + 1) \sbra*{2 ((n + 1) q - p) - (n - p + q + 2) \delta_{\Gamma}}
			\tensor{D}{_{\gamma}^{\gamma}} \Hproduct{\omega}{\omega} \\
		&\geq (2 q - n - 2) \delta_{\Gamma} \Hproduct{D \, \sharp \, \omega}{\omega} \\
		&\quad + (n - q + 1) \sbra*{2 ((n + 1) q - p) - (n - p + q + 2) \delta_{\Gamma}}
			\tensor{D}{_{\gamma}^{\gamma}} \Hproduct{\omega}{\omega}.
	\end{align}

	If $2 q \leq n + 2$,
	then \cref{eq:estimate-of-D} yields
	\begin{align}
		&\Hproduct{- n R \, \sharp \, \overline{\sharp} \, \omega
			- (q - 1) \Ric \sharp \, \omega
			- (n - q + 1) \Ric \overline{\sharp} \, \omega}{\omega} \\
		&\geq (n - q + 1) \sbra*{2 ((n + 1) q - p) - (n - p + q + 2) \delta_{\Gamma}}
			\tensor{D}{_{\gamma}^{\gamma}} \Hproduct{\omega}{\omega}.
	\end{align}
	Thus we have
	\begin{equation}
		0
		\geq (n - q + 1) \sbra*{2 ((n + 1) q - p) - (n - p + q + 2) \delta_{\Gamma}}
			\int_{M_{\Gamma}} \tensor{D}{_{\gamma}^{\gamma}} \Hproduct{\omega}{\omega}
			\, \theta_{\Gamma} \wedge (d \theta_{\Gamma})^{n}.
	\end{equation}
	Now it follows from $\delta_{\Gamma} < m_{p, q}$ that
	\begin{equation}
		(n - q + 1) \sbra*{2 ((n + 1) q - p) - (n - p + q + 2) \delta_{\Gamma}}
		> 0.
	\end{equation}
	Moreover,
	$\tensor{D}{_{\gamma}^{\gamma}}$ is positive on an open dense subset by \cref{lem:trace-of-D}.
	Therefore $\Hproduct{\omega}{\omega}$ is equal to zero on an open dense subset,
	which implies $\omega = 0$ by the continuity.

	If $2 q \geq n + 2$,
	then \cref{eq:estimate-of-D} yields
	\begin{align}
		&\Hproduct{- n R \, \sharp \, \overline{\sharp} \, \omega
			- (q - 1) \Ric \sharp \, \omega
			- (n - q + 1) \Ric \overline{\sharp} \, \omega}{\omega} \\
		&\geq (2 q - n - 2) \delta_{\Gamma}
			\Hproduct{\tensor{D}{_{\gamma}^{\gamma}} \omega + D \, \sharp \, \omega}{\omega} \\
		&\quad + \sbra*{2 ((n + 1) q - p) (n - q + 1) - ((n - p + q) (n - q + 1) + n) \delta_{\Gamma}} 
			\tensor{D}{_{\gamma}^{\gamma}} \Hproduct{\omega}{\omega} \\
		&\geq \sbra*{2 ((n + 1) q - p) (n - q + 1) - ((n - p + q) (n - q + 1) + n) \delta_{\Gamma}} 
			\tensor{D}{_{\gamma}^{\gamma}} \Hproduct{\omega}{\omega}.
	\end{align}
	Thus we have
	\begin{equation}
		0
		\geq \sbra*{2 ((n + 1) q - p) (n - q + 1) - ((n - p + q) (n - q + 1) + n) \delta_{\Gamma}}
			\int_{M_{\Gamma}} \tensor{D}{_{\gamma}^{\gamma}} \Hproduct{\omega}{\omega}
			\, \theta_{\Gamma} \wedge (d \theta_{\Gamma})^{n}.
	\end{equation}
	Now it follows from $\delta_{\Gamma} < m_{p, q}$ that
	\begin{equation}
		2 ((n + 1) q - p) (n - q + 1) - ((n - p + q) (n - q + 1) + n) \delta_{\Gamma}
		> 0.
	\end{equation}
	Moreover,
	$\tensor{D}{_{\gamma}^{\gamma}}$ is positive on an open dense subset by \cref{lem:trace-of-D}.
	Therefore $\Hproduct{\omega}{\omega}$ is equal to zero on an open dense subset,
	which implies $\omega = 0$ by the continuity.
\end{proof}

\section{Concluding remarks}
\label{section:concluding-remarks}

In our main theorems,
we study $H^{p, q}_{\KR}(M_{\Gamma})$ with $p + q \leq n - 1$
since the \Weitzenbock-type formulae (\cref{prop:Weitzenbock-type-formula}) have been proved only for these degrees.
If $p + q = n$,
then we need to consider fourth order differential operator~\cite{Case2021-Rumin-preprint}*{Definition 13.1},
and there exist no known \Weitzenbock-type formulae.
However,
we can show that $H^{p, q}_{\KR}(M_{\Gamma}) = 0$ even for $p + q = n$
if $\Gamma$ is convex cocompact and $\delta_{\Gamma} < 2$.

\begin{proposition}
\label{prop:vanishing-KR-cohomology-with-small-delta}
	Let $\Gamma$ be a torsion-free convex cocompact discrete subgroup of $PU(n + 1, 1)$
	with $\delta_{\Gamma} < 2$.
	Then $H^{p, q}_{\KR}(M_{\Gamma}) = 0$
	for any $0 \leq p \leq n + 1$ and $1 \leq q \leq n - 1$.
\end{proposition}

\begin{proof}
	It follows from \cite{Yau1981}*{Theorem B} that
	the dimension of the Kohn-Rossi cohomology $H^{p, q}_{\KR}(M_{\Gamma})$
	for $0 \leq p \leq n + 1$ and $1 \leq q \leq n - 1$
	is given by the sum of local invariants at the singularities of a Stein space
	that bounds $M_{\Gamma}$.
	On the other hand,
	\cite{Dey-Kapovich2020}*{Theorem 1} implies that
	$X_{\Gamma}$ is a Stein manifold;
	in particular,
	$X_{\Gamma}$ has no singularities and bounds $M_{\Gamma}$.
	Thus we have $H^{p, q}_{\KR}(M_{\Gamma}) = 0$ for any $0 \leq p \leq n + 1$ and $1 \leq q \leq n - 1$.
\end{proof}

\begin{example}
	Let $\Gamma$ be a torsion-free discrete subgroup of $PU(n + 1, 1)$ generated by the matrix
	\begin{equation}
		\begin{pmatrix}
			\sqrt{2} & - 1 & 0 \\
			- 1 & \sqrt{2} & 0 \\
			0 & 0 & I_{n}
		\end{pmatrix}
		.
	\end{equation}
	The limit set $\Lambda_{\Gamma}$ of $\Gamma$ is $\{(\pm 1, 0, \dots , 0)\}$
	and its closed convex hull $C_{\Gamma}$ is given by $\opop{- 1}{1} \times \{0\}$.
	In particular,
	$\Gamma$ is elementary and $\delta_{\Gamma} = 0$.
	We can see that $C_{\Gamma} / \Gamma$ is compact,
	which means that $\Gamma$ is convex cocompact.
	\cref{prop:vanishing-KR-cohomology-with-small-delta} yields that
	$H^{p, q}_{\KR}(M_{\Gamma}) = 0$ for any $0 \leq p \leq n + 1$ and $1 \leq q \leq n - 1$.
	Note that $M_{\Gamma}$ coincides with the CR Hopf manifold,
	and this result is an improvement of \cite{Case2021-Rumin-preprint}*{Example 17.6}.
\end{example}

We also add a comment on the condition for $\delta_{\Gamma}$.
We need to impose $\delta_{\Gamma} < n$
in \cref{thm:vanishing-of-KR-cohomology-1,thm:vanishing-of-KR-cohomology-2,prop:vanishing-KR-cohomology-with-small-delta}.
It is natural to ask what happens when $\delta_{\Gamma} \geq n$.
The following proposition shows that
the Kohn-Rossi cohomology may vanish even when $\delta_{\Gamma} = n$.

\begin{proposition}
\label{prop:vanishing-KR-cohomology-with-n}
	For each $n \geq 2$,
	there exists a torsion-free convex cocompact discrete subgroup $\Gamma$ of $PU(n + 1, 1)$
	such that $\delta_{\Gamma} = n$ and
	$H^{p, q}_{\KR}(M_{\Gamma}) = 0$ for any $0 \leq p \leq n + 1$ and $1 \leq q \leq n - 1$.
\end{proposition}

\begin{proof}
	Similar to the unitary group $U(n + 1, 1)$,
	the orthogonal group $O(n + 1, 1)$ with respect to $\diag(- 1, 1, \dots , 1)$
	acts on the ball
	\begin{equation}
		\ball_{\bbR}^{n + 1}
		\coloneqq \Set{x = (x^{1}, \dots , x^{n + 1}) \in \bbR^{n + 1} | \abs{x}^{2} < 1}
	\end{equation}
	by the fractional linear transformation,
	which preserves the real hyperbolic metric on $\ball_{\bbR}^{n + 1}$.

	Take $\Gamma_{0}$ be a torsion-free discrete subgroup of $O(n + 1, 1)$
	so that $\ball_{\bbR}^{n + 1} / \Gamma_{0}$ is compact.
	Denote by $\Gamma$ the image of $\Gamma_{0}$ under the composition
	of the inclusion $O(n + 1, 1) \to U(n + 1, 1)$ and the projection $U(n + 1, 1) \to PU(n + 1, 1)$.
	This $\Gamma$ is a torsion-free discrete subgroup of $PU(n + 1, 1)$ and
	\begin{equation}
		\Lambda_{\Gamma}
		= \Set{x \in \bbR^{n + 1} | \abs{x}^{2} = 1}.
	\end{equation}
	It follows from the compactness of $\ball_{\bbR}^{n + 1} / \Gamma_{0}$
	that $\Gamma$ is convex cocompact and $\delta_{\Gamma} = \dim_{H} \Lambda_{\Gamma} = n$.
	Moreover,
	$X_{\Gamma}$ is a Stein manifold~\cite{Burns-Shnider1976}*{Proposition 6.4}.
	We obtain from \cite{Yau1981}*{Theorem B} that
	$H^{p, q}_{\KR}(M_{\Gamma}) = 0$ for any $0 \leq p \leq n + 1$ and $1 \leq q \leq n - 1$.
\end{proof}

\section*{Acknowledgements}

The author would like to thank Wei Wang for some helpful comments on \cref{rem:scalar-curvature}.
He is also grateful to the anonymous referee for valuable suggestions
which led to improvements of the revised version of the paper.

\bibliography{my-reference,my-reference-preprint}

\begin{bibdiv}
\begin{biblist}

\bib{Boutet_de_Monvel1975}{incollection}{
      author={Boutet~de Monvel, L.},
       title={Int\'egration des \'equations de {C}auchy-{R}iemann induites
  formelles},
        date={1975},
   booktitle={S\'eminaire {G}oulaouic-{L}ions-{S}chwartz 1974--1975;
  \'equations aux deriv\'ees partielles lin\'eaires et non lin\'eaires},
   publisher={Centre Math., \'Ecole Polytech., Paris},
       pages={Exp. No. 9, 14},
      review={\MR{0409893}},
}

\bib{Beardon1968}{article}{
      author={Beardon, A.~F.},
       title={The exponent of convergence of {P}oincar\'{e} series},
        date={1968},
        ISSN={0024-6115,1460-244X},
     journal={Proc. London Math. Soc. (3)},
      volume={18},
       pages={461\ndash 483},
         url={https://doi.org/10.1112/plms/s3-18.3.461},
      review={\MR{227402}},
}

\bib{Bowditch1995}{article}{
      author={Bowditch, B.~H.},
       title={Geometrical finiteness with variable negative curvature},
        date={1995},
        ISSN={0012-7094,1547-7398},
     journal={Duke Math. J.},
      volume={77},
      number={1},
       pages={229\ndash 274},
         url={https://doi.org/10.1215/S0012-7094-95-07709-6},
      review={\MR{1317633}},
}

\bib{Burns-Shnider1976}{article}{
      author={Burns, D., Jr.},
      author={Shnider, S.},
       title={Spherical hypersurfaces in complex manifolds},
        date={1976},
        ISSN={0020-9910},
     journal={Invent. Math.},
      volume={33},
      number={3},
       pages={223\ndash 246},
         url={https://doi.org/10.1007/BF01404204},
      review={\MR{0419857}},
}

\bib{Case2021-Rumin-preprint}{unpublished}{
      author={Case, Jeffrey~S.},
       title={The bigraded {R}umin complex via differential forms},
        date={2021},
        note={\texttt{arXiv:2108.13911}, to appear in Mem. Amer. Math. Soc.},
}

\bib{Corlette-Iozzi1999}{article}{
      author={Corlette, Kevin},
      author={Iozzi, Alessandra},
       title={Limit sets of discrete groups of isometries of exotic hyperbolic
  spaces},
        date={1999},
        ISSN={0002-9947,1088-6850},
     journal={Trans. Amer. Math. Soc.},
      volume={351},
      number={4},
       pages={1507\ndash 1530},
         url={https://doi.org/10.1090/S0002-9947-99-02113-3},
      review={\MR{1458321}},
}

\bib{Chern-Moser1974}{article}{
      author={Chern, S.~S.},
      author={Moser, J.~K.},
       title={Real hypersurfaces in complex manifolds},
        date={1974},
        ISSN={0001-5962},
     journal={Acta Math.},
      volume={133},
       pages={219\ndash 271},
         url={https://doi.org/10.1007/BF02392146},
      review={\MR{0425155}},
}

\bib{Corlette1990}{article}{
      author={Corlette, Kevin},
       title={Hausdorff dimensions of limit sets. {I}},
        date={1990},
        ISSN={0020-9910,1432-1297},
     journal={Invent. Math.},
      volume={102},
      number={3},
       pages={521\ndash 541},
         url={https://doi.org/10.1007/BF01233439},
      review={\MR{1074486}},
}

\bib{Dey-Kapovich2020}{article}{
      author={Dey, Subhadip},
      author={Kapovich, Michael},
       title={A note on complex-hyperbolic {K}leinian groups},
        date={2020},
        ISSN={2199-6792,2199-6806},
     journal={Arnold Math. J.},
      volume={6},
      number={3-4},
       pages={397\ndash 406},
         url={https://doi.org/10.1007/s40598-020-00143-x},
      review={\MR{4181718}},
}

\bib{Folland-Stein1974}{article}{
      author={Folland, G.~B.},
      author={Stein, E.~M.},
       title={Estimates for the {$\bar \partial \sb{b}$} complex and analysis
  on the {H}eisenberg group},
        date={1974},
        ISSN={0010-3640},
     journal={Comm. Pure Appl. Math.},
      volume={27},
       pages={429\ndash 522},
         url={https://doi.org/10.1002/cpa.3160270403},
      review={\MR{0367477}},
}

\bib{Garfield2001}{thesis}{
      author={Garfield, Peter~McKee},
       title={The bigraded {R}umin complex},
        type={Ph.D. Thesis},
   publisher={ProQuest LLC, Ann Arbor, MI},
        date={2001},
        ISBN={978-0493-49492-0},
  url={http://gateway.proquest.com/openurl?url_ver=Z39.88-2004&rft_val_fmt=info:ofi/fmt:kev:mtx:dissertation&res_dat=xri:pqdiss&rft_dat=xri:pqdiss:3036468},
        note={Thesis (Ph.D.)--University of Washington},
      review={\MR{2702802}},
}

\bib{Garfield-Lee1998}{incollection}{
      author={Garfield, Peter~M.},
      author={Lee, John~M.},
       title={The {R}umin complex on {CR} manifolds},
        date={1998},
       pages={29\ndash 36},
        note={CR geometry and isolated singularities (Japanese) (Kyoto, 1996)},
      review={\MR{1660485}},
}

\bib{Goldman1999}{book}{
      author={Goldman, William~M.},
       title={Complex hyperbolic geometry},
      series={Oxford Mathematical Monographs},
   publisher={The Clarendon Press, Oxford University Press, New York},
        date={1999},
        ISBN={0-19-853793-X},
        note={Oxford Science Publications},
      review={\MR{1695450}},
}

\bib{Izeki2002}{article}{
      author={Izeki, Hiroyasu},
       title={Convex-cocompactness of {K}leinian groups and conformally flat
  manifolds with positive scalar curvature},
        date={2002},
        ISSN={0002-9939,1088-6826},
     journal={Proc. Amer. Math. Soc.},
      volume={130},
      number={12},
       pages={3731\ndash 3740},
         url={https://doi.org/10.1090/S0002-9939-02-06504-8},
      review={\MR{1920055}},
}

\bib{Kapovich2022}{incollection}{
      author={Kapovich, Michael},
       title={A survey of complex hyperbolic {K}leinian groups},
        date={2022},
   booktitle={In the tradition of {T}hurston {II}. {G}eometry and groups},
   publisher={Springer, Cham},
       pages={7\ndash 51},
         url={https://doi.org/10.1007/978-3-030-97560-9_2},
      review={\MR{4472047}},
}

\bib{Kohn1965}{incollection}{
      author={Kohn, J.~J.},
       title={Boundaries of complex manifolds},
        date={1965},
   booktitle={Proc. {C}onf. {C}omplex {A}nalysis ({M}inneapolis, 1964)},
   publisher={Springer, Berlin},
       pages={81\ndash 94},
      review={\MR{0175149}},
}

\bib{Kohn-Rossi1965}{article}{
      author={Kohn, J.~J.},
      author={Rossi, Hugo},
       title={On the extension of holomorphic functions from the boundary of a
  complex manifold},
        date={1965},
        ISSN={0003-486X},
     journal={Ann. of Math. (2)},
      volume={81},
       pages={451\ndash 472},
         url={https://doi.org/10.2307/1970624},
      review={\MR{0177135}},
}

\bib{Nayatani1997}{article}{
      author={Nayatani, Shin},
       title={Patterson-{S}ullivan measure and conformally flat metrics},
        date={1997},
        ISSN={0025-5874,1432-1823},
     journal={Math. Z.},
      volume={225},
      number={1},
       pages={115\ndash 131},
         url={https://doi.org/10.1007/PL00004301},
      review={\MR{1451336}},
}

\bib{Nayatani1999}{incollection}{
      author={Nayatani, Shin},
       title={Discrete groups of complex hyperbolic isometries and
  pseudo-{H}ermitian structures},
        date={1999},
   booktitle={Analysis and geometry in several complex variables ({K}atata,
  1997)},
      series={Trends Math.},
   publisher={Birkh\"{a}user Boston, Boston, MA},
       pages={209\ndash 237},
      review={\MR{1699844}},
}

\bib{Patterson1976}{article}{
      author={Patterson, S.~J.},
       title={The limit set of a {F}uchsian group},
        date={1976},
        ISSN={0001-5962,1871-2509},
     journal={Acta Math.},
      volume={136},
      number={3-4},
       pages={241\ndash 273},
         url={https://doi.org/10.1007/BF02392046},
      review={\MR{450547}},
}

\bib{Rumin1994}{article}{
      author={Rumin, Michel},
       title={Formes diff\'{e}rentielles sur les vari\'{e}t\'{e}s de contact},
        date={1994},
        ISSN={0022-040X},
     journal={J. Differential Geom.},
      volume={39},
      number={2},
       pages={281\ndash 330},
         url={http://projecteuclid.org/euclid.jdg/1214454873},
      review={\MR{1267892}},
}

\bib{Sullivan1979}{article}{
      author={Sullivan, Dennis},
       title={The density at infinity of a discrete group of hyperbolic
  motions},
        date={1979},
        ISSN={0073-8301,1618-1913},
     journal={Inst. Hautes \'{E}tudes Sci. Publ. Math.},
      number={50},
       pages={171\ndash 202},
         url={http://www.numdam.org/item?id=PMIHES_1979__50__171_0},
      review={\MR{556586}},
}

\bib{Tanaka1975}{book}{
      author={Tanaka, Noboru},
       title={A differential geometric study on strongly pseudo-convex
  manifolds},
      series={Lectures in Mathematics, Department of Mathematics, Kyoto
  University, No. 9},
   publisher={Kinokuniya Book Store Co., Ltd., Tokyo},
        date={1975},
      review={\MR{0399517}},
}

\bib{Wang2003}{article}{
      author={Wang, Wei},
       title={Canonical contact forms on spherical {CR} manifolds},
        date={2003},
        ISSN={1435-9855},
     journal={J. Eur. Math. Soc. (JEMS)},
      volume={5},
      number={3},
       pages={245\ndash 273},
         url={https://doi.org/10.1007/s10097-003-0050-8},
      review={\MR{2002214}},
}

\bib{Yau1981}{article}{
      author={Yau, Stephen S.~T.},
       title={Kohn-{R}ossi cohomology and its application to the complex
  {P}lateau problem. {I}},
        date={1981},
        ISSN={0003-486X},
     journal={Ann. of Math. (2)},
      volume={113},
      number={1},
       pages={67\ndash 110},
         url={https://doi.org/10.2307/1971134},
      review={\MR{604043}},
}

\bib{Yue1999}{inproceedings}{
      author={Yue, Chengbo},
       title={Webster curvature and hausdorff dimension of complex hyperbolic
  kleinian groups},
organization={World Science Publication Singapore},
        date={1999},
   booktitle={Dynamical systems. proceedings of the international conference in
  honor of professor liao shantao},
       pages={319\ndash 328},
}

\end{biblist}
\end{bibdiv}

\end{document}